\documentclass[a4paper, twoside]{amsart}

\usepackage{ifpdf}
\ifpdf
  \usepackage[pdftex]{color,graphicx}
\else
  \usepackage[dvips]{graphicx}
\fi

\usepackage{verbatim}
\usepackage[british]{babel}
\usepackage{hyperref}
\usepackage{amsmath}
\usepackage{amssymb}
\usepackage{amsthm}
\usepackage{amsfonts}
\usepackage{amsxtra}
\usepackage[all]{xy}
\input xy
\xyoption{all}
\xyoption{arc}

\usepackage{tabularx}
\usepackage{url}
\usepackage{fancyhdr}
\usepackage{latexsym}

\usepackage{color,graphicx}
\usepackage{pstricks}
\usepackage{pst-plot}
\usepackage{pstricks-add}
\usepackage{pst-eps}
\usepackage{epsfig}

\renewcommand\epsilon\varepsilon

\newcommand\diag{\operatorname{diag}}

\newcommand\bo{\mathcal{B}}

\newcommand\Tr{\operatorname{Tr}}

\newcommand\conv{\operatorname{conv}}

\newcommand\bbR{\mathbb{R}}
\newcommand\bbC{\mathbb{C}}
\newcommand\apschur{\mathrm{AP}_{p,\mathrm{cb}}^{\mathrm{Schur}}}
\newcommand\lambdaapschur{\Lambda_{p,\mathrm{cb}}^{\mathrm{Schur}}}

\newcommand\lra{\longrightarrow}
\DeclareMathOperator\GL{GL}

\DeclareMathOperator\SL{SL}

\DeclareMathOperator\Sp{Sp}
\DeclareMathOperator\U{U}

\DeclareMathOperator\SO{SO}
\DeclareMathOperator\SU{SU}

\DeclareMathOperator\rr{Rank_{\mathbb{R}}}

\DeclareMathOperator\id{id}

\theoremstyle{definition}
\newtheorem{thm}{Theorem}[section]
\newtheorem{dfn}[thm]{Definition}

\newtheorem{lem}[thm]{Lemma}
\newtheorem{prp}[thm]{Proposition}

\newtheorem{rmk}[thm]{Remark}

\author{Tim de Laat}
\address{Department of Mathematical Sciences, University of Copenhagen,
\newline Universitetsparken 5, DK-2100 Copenhagen \O, Denmark}
\email{tlaat@math.ku.dk}
\thanks{The author is supported by the Danish National Research Foundation through the Centre for Symmetry and Deformation.}

\title[Approximation properties for noncommutative $L^p$-spaces]{Approximation properties for noncommutative $L^p$-spaces associated with lattices in Lie groups}

\begin{document}

\maketitle

\begin{abstract}
In 2010, Lafforgue and de la Salle gave examples of noncommutative $L^p$-spaces without the operator space approximation property (OAP) and, hence, without the completely bounded approximation property (CBAP). To this purpose, they introduced the property of completely bounded approximation by Schur multipliers on $S^p$, denoted $\apschur$, and proved that for $p \in [1,\frac{4}{3}) \cup (4,\infty]$ the groups $\SL(n,\mathbb{Z})$, with $n \geq 3$, do not have the $\apschur$. Since for $p \in (1,\infty)$ the $\apschur$ is weaker than the approximation property of Haagerup and Kraus (AP), these groups were also the first examples of exact groups without the AP. Recently, Haagerup and the author proved that also the group $\Sp(2,\bbR)$ does not have the AP, without using the $\apschur$. In this paper, we prove that $\Sp(2,\bbR)$ does not have the $\apschur$ for $p \in [1,\frac{12}{11}) \cup (12,\infty]$. It follows that a large class of noncommutative $L^p$-spaces does not have the OAP or CBAP.
\end{abstract}

\section{Introduction} \label{sec:introduction}
Let $M$ be a finite von Neumann algebra with normal faithful trace $\tau$. For $1 \leq p < \infty$, the noncommutative $L^p$-space $L^p(M,\tau)$ is defined as the completion of $M$ with respect to the norm $\|x\|_p=\tau((x^*x)^{\frac{p}{2}})^{\frac{1}{p}}$, and for $p=\infty$, we put $L^{\infty}(M,\tau)=M$ with operator norm. In \cite{kosaki}, Kosaki showed that noncommutative $L^p$-spaces can be realized by interpolating between $M$ and $L^1(M,\tau)$. This leads to an operator space structure on them, as described by Pisier \cite{pisieroh} (see also \cite{jungeruan}).

An operator space $E$ is said to have the completely bounded approximation property (CBAP) if there exists a net $F_{\alpha}$ of finite-rank maps on $E$ such that $\sup_{\alpha}\|F_{\alpha}\|_{cb} < C$ for some $C > 0$, and $\lim_{\alpha} \|F_{\alpha}x-x\|=0$ for every $x \in E$. The infimum of all possible $C$'s is denoted by $\Lambda(E)$. If $\Lambda(E)=1$, we say that $E$ has the completely contractive approximation property (CCAP). An operator space $E$ is said to have the operator space approximation property (OAP) if there exists a net $F_{\alpha}$ of finite-rank maps on $E$ such that $\lim_{\alpha} \|(\id_{\mathcal{K}(\ell^2)} \otimes F_{\alpha})x-x\|=0$ for all $x \in \mathcal{K}(\ell^2) \otimes_{\min} E$. Here $\mathcal{K}(\ell^2)$ denotes the space of compact operators on the Hilbert space $\ell^2$. The CBAP goes back to De Canni\`ere and Haagerup \cite{decannierehaagerup}, and the OAP was defined by Effros and Ruan \cite{effrosruanap}. By definition, the CCAP implies the CBAP, which in turn implies the OAP.

Recall that a lattice in a Lie group $G$ is a discrete subgroup $\Gamma$ of $G$ such that $G \slash \Gamma$ has finite invariant measure. In this paper, we consider noncommutative $L^p$-spaces of the form $L^p(L(\Gamma))$, where $L(\Gamma)$ is the group von Neumann algebra of a lattice $\Gamma$ in a connected simple Lie group $G$. Such a von Neumann algebra $L(\Gamma)$ is finite and has canonical trace $\tau:x \mapsto \langle x\delta_1,\delta_1 \rangle$, where $\delta_1 \in \ell^2(\Gamma)$ is the characteristic function of the unit element $1 \in \Gamma$.

It was proved by Junge and Ruan \cite[Proposition 3.5]{jungeruan} that if $\Gamma$ is a weakly amenable (countable) discrete group, then for $p \in (1,\infty)$, the noncommutative $L^p$-space $L^p(L(\Gamma))$ has the CBAP. Recall that connected simple Lie groups of real rank zero are amenable. By the work of Cowling and Haagerup \cite{cowlinghaagerup} and Hansen \cite{hansen}, all connected simple Lie groups of real rank one are weakly amenable. This implies that for every $p \in (1,\infty)$ and every lattice $\Gamma$ in a connected simple Lie group $G$ of real rank zero or one, the noncommutative $L^p$-space $L^p(L(\Gamma))$ has the CBAP.

The existence of noncommutative $L^p$-spaces without the CBAP follows from the work of Szankowski \cite{szankowski}. The first concrete examples were given recently by Lafforgue and de la Salle \cite{ldls}. They proved that for all $p \in [1,\frac{4}{3}) \cup (4,\infty]$ and all lattices $\Gamma$ in $\SL(n,\bbR)$, where $n \geq 3$, the space $L^p(L(\Gamma))$ does not have the OAP (or CBAP). They also proved analogous results for lattices in Lie groups over nonarchimedean fields. In their work, the failure of the OAP for the aforementioned noncommutative $L^p$-spaces follows from the failure of a certain approximation property for the groups $\SL(n,\bbR)$. This property, called the property of completely bounded approximation by Schur multipliers on $S^p$ (see Section \ref{sec:apschur}), denoted $\apschur$, was introduced by Lafforgue and de la Salle exactly to this purpose.

Other approximation properties for groups (see \cite{brownozawa}), e.g., amenability, weak amenability, and the approximation property of Haagerup and Kraus (AP) (see \cite{haagerupkraus}), are related to the $\apschur$. It is well-known that amenability of a group $G$ (strictly) implies weak amenability, which in turn (strictly) implies the AP. For $p \in (1,\infty)$, the $\apschur$ is weaker than the AP. In this way, the $\apschur$ gave rise to the first example of an exact group without the AP, namely $\SL(3,\mathbb{Z})$. Recently, Haagerup and the author proved that also $\Sp(2,\bbR)$ does not have the AP \cite{haagerupdelaat1}, in a more direct way than Lafforgue and de la Salle did for $\SL(3,\bbR)$. Indeed, the $\apschur$ was not used in the proof. On the other hand, as was mentioned earlier, the method of Lafforgue and de la Salle also gives information about approximation properties of certain noncommutative $L^p$-spaces. For this, it is actually crucial to use the $\apschur$. Haagerup and the author also proved that all connected simple Lie groups with finite center and real rank greater than or equal to two do not have the AP, building on the failure of the AP for both $\SL(3,\bbR)$ and $\Sp(2,\bbR)$.

The following are the main results of this article.
\newtheorem*{thm:sp2notapschur}{Theorem \ref{thm:sp2notapschur}}
\begin{thm:sp2notapschur}
  For $p \in [1,\frac{12}{11}) \cup (12,\infty]$, the group $\Sp(2,\bbR)$ does not have the $\apschur$.
\end{thm:sp2notapschur}
\newtheorem*{thm:maintheoremnclp}{Theorem \ref{thm:maintheoremnclp}}
\begin{thm:maintheoremnclp}
Let $p \in [1,\frac{12}{11}) \cup (12,\infty]$, and let $\Gamma$ be a lattice in a connected simple Lie group with finite center and real rank greater than or equal to two. Then $L^p(L(\Gamma))$ does not have OAP (or CBAP).
\end{thm:maintheoremnclp}                                                                                                                                                                                                                                                                                                                                                                                                                                                                                                                                           
The paper is organized as follows. In Section \ref{sec:preliminaries}, we recall some preliminary results, and we make a study of Schur multipliers on Schatten classes corresponding to (compact) Gelfand pairs, which provides us with suitable tools for our proof. In Section \ref{sec:sp2}, we prove Theorem \ref{thm:sp2notapschur}, and in Section \ref{sec:general}, we prove Theorem \ref{thm:maintheoremnclp}.

\section{Preliminaries} \label{sec:preliminaries}
\subsection{Schur multipliers on Schatten classes} \label{sec:smosc}
This section partly follows the exposition of \cite[Section 1]{ldls}. More details can be found there.

For $p \in [1,\infty]$ and a (separable) Hilbert space $\mathcal{H}$, let $S^p(\mathcal{H})$ denote the $p^{\textrm{th}}$ Schatten class on $\mathcal{H}$. Recall that $S^{\infty}(\mathcal{H})$ is the Banach space $\mathcal{K}(\mathcal{H})$ of compact operators (with operator norm) on $\mathcal{H}$, and for $p \in [1,\infty)$, the space $S^p(\mathcal{H})$ consists of the operators $T$ on $\mathcal{H}$ such that $\|T\|_p=\Tr((T^*T)^{\frac{p}{2}})^{\frac{1}{p}} < \infty$, where $\Tr$ denotes the (semifinite) trace on $\bo(\mathcal{H})$. In this way, $S^p(\mathcal{H})$ is a Banach space for all $p \in [1,\infty]$. We use the notation $S_n^p=S^p(\ell^2_n)$ and $S^p=S^p(\ell^2)$. Note that the space $S^2(\mathcal{H})$ corresponds to the Hilbert-Schmidt operators on $\mathcal{H}$.

Schatten classes can be realized by interpolating between certain noncommutative $L^p$-spaces in the semifinite setting. Indeed, we have $S^p(\mathcal{H})=L^p(\mathcal{B}(\mathcal{H}),\Tr)$. Noncommutative $L^p$-spaces in the semifinite setting can be defined analogously to the finite case, which was described in Section \ref{sec:introduction}. For details, see \cite{pisiernoncomvvlpsp}. The natural operator space structure on $S^p(\mathcal{H})$ follows from \cite{pisieroh}. For our purposes, the following characterization of the completely bounded norm of a linear map between Schatten classes is important. Recall that $S^p(\mathcal{H}) \otimes S^p(\mathcal{K})$ (algebraic tensor product) embeds naturally into $S^p(\mathcal{H} \otimes \mathcal{K})$ (Hilbert space tensor product). Let $T:S^p(\mathcal{H}) \lra S^p(\mathcal{H})$ be a bounded linear map, and let $\mathcal{K}=\ell^2$. Then $T$ is completely bounded if the map $T \otimes \mathrm{id}_{S^p}$ extends to a bounded linear map on $S^p(\mathcal{H} \otimes \ell^2)$, and we have $\|T\|_{cb}=\|T \otimes \id_{S^p}\|=\sup_{n \in \mathbb{N}} \|T \otimes \id_{S^p_n}\|$ (see \cite[Lemma 1.7]{pisiernoncomvvlpsp}).

A linear map $T:M_n(\bbC) \lra M_n(\bbC)$ of the form $[x_{ij}] \mapsto [\psi_{ij}x_{ij}]$ for some matrix $\psi \in M_n(\mathbb{C})$ is called a Schur multiplier on $M_n(\bbC)$. More precisely, the operator $T$ is called the Schur multiplier on $M_n(\mathbb{C})$ with symbol $\psi$, and it is also denoted by $M_{\psi}$. In what follows, we need more general notions of Schur multipliers.

Let $(X,\mu)$ be a $\sigma$-finite measure space. Let $k \in L^2(X \times X,\mu \otimes \mu)$. It is well-known that the map $T_k:L^2(X,\mu) \lra L^2(X,\mu)$ defined by $(T_kf)(x)=\int_X k(x,y)f(y)d\mu(y)$, is a Hilbert-Schmidt operator on $L^2(X,\mu)$. Conversely, if $T \in S^2(L^2(X,\mu))$, then $T=T_k$ for some $k \in L^2(X \times X,\mu \otimes \mu)$. In this way, we can identify $S^2(L^2(X,\mu))$ with $L^2(X \times X,\mu \otimes \mu)$, and we see that every Schur multiplier on $S^2(L^2(X,\mu))$ comes from a function $\psi \in L^{\infty}(X \times X,\mu \otimes \mu)$ acting by multiplication on $L^2(X \times X,\mu \otimes \mu)$.
\begin{dfn} \label{dfn:multiplier}
  Let $p \in [1,\infty]$, and let $\psi \in L^{\infty}(X \times X,\mu \otimes \mu)$. The Schur multiplier with symbol $\psi$ is said to be bounded (resp.~completely bounded) on $S^p(L^2(X,\mu))$ if it maps $S^p(L^2(X,\mu)) \cap S^2(L^2(X,\mu))$ into $S^p(L^2(X,\mu))$ (by $T_k \mapsto T_{\psi k}$), and if this map extends (necessarily uniquely) to a bounded (resp.~completely bounded) map $M_{\psi}$ on $S^p(L^2(X,\mu))$.
\end{dfn}
The norm of such a bounded multiplier $\psi$ is defined by $\|\psi\|_{MS^p(L^2(X,\mu))}=\|M_{\psi}\|$, and its completely bounded norm by $\|\psi\|_{cbMS^p(L^2(X,\mu))}=\|M_{\psi}\|_{cb}$. The spaces of multipliers and completely bounded multipliers are denoted by $MS^p(L^2(X,\mu))$ and $cbMS^p(L^2(X,\mu))$, respectively. It follows that for every $p \in [1,\infty]$ and $\psi \in L^{\infty}(X \times X,\mu \otimes \mu)$, we have $\|\psi\|_{\infty} \leq \|\psi\|_{MS^p(L^2(X,\mu))} \leq \|\psi\|_{cbMS^p(L^2(X,\mu))}$.

If $\frac{1}{p}+\frac{1}{q}=1$, we have $\|\psi\|_{MS^p(L^2(X,\mu))}=\|\psi\|_{MS^q(L^2(X,\mu))}$. By interpolation and duality we have that whenever $2 \leq p \leq q \leq \infty$, then $\|\psi\|_{MS^p(L^2(X,\mu))} \leq \|\psi\|_{MS^q(L^2(X,\mu))}$. These results also hold for the completely bounded norm.
\begin{lem}(\cite[Lemma 1.5 and Remark 1.6]{ldls}) \label{lem:schurbcb}
The Schur multiplier corresponding to $\psi \in L^{\infty}(X \times X,\mu \otimes \mu)$ is completely bounded on $S^p(L^2(X,\mu))$ if and only if the Schur multiplier corresponding to $\tilde{\psi}(x,\xi,y,\eta)=\psi(x,y)$ is completely bounded on $S^p(L^2(X \times \Omega,\mu \otimes \nu))$, where $(\Omega,\nu)$ is a $\sigma$-finite measure space, and
\[
  \|\psi\|_{cbMS^p(L^2(X,\mu))}=\|\tilde{\psi}\|_{cbMS^p(L^2(X \times \Omega,\mu \otimes \nu))}.
\]
If $L^2(\Omega,\nu)$ is infinite-dimensional, these norms equal $\|\tilde{\psi}\|_{MS^p(L^2(X \times \Omega,\mu \otimes \nu))}$.
\end{lem}
\begin{lem}(\cite[Theorem 1.19]{ldls}) \label{lem:finitecharacterization}
  Let $(X,\mu)$ be a locally compact space with a $\sigma$-finite Radon measure $\mu$, and let $\psi:X \times X \lra \bbC$ be a bounded continuous function. Let $1 \leq p \leq \infty$. The following are equivalent:
\begin{enumerate}
  \item we have $\psi \in MS^p(L^2(X,\mu))$ with $\|\psi\|_{MS^p(L^2(X,\mu))} \leq C$,
  \item for every finite set $F=\{x_1,\ldots,x_n\} \subset X$ such that $F \subset \mathrm{supp}(\mu)$, the Schur multiplier given by $(\psi(x_i,x_j))_{i,j}$ is bounded on $S^p(\ell^2(F))$ with norm smaller than or equal to $C$.
\end{enumerate}
  The analogous statement holds in the completely bounded case. In particular, the norm and the completely bounded norm of the multiplier only depend on the support of $\mu$, and if this support does not have any isolated points, then the norm and the completely bounded norm coincide.
\end{lem}

\subsection{Schur multipliers on locally compact groups} \label{subsec:mloccpt}
For a locally compact group $G$ and a function $\varphi \in L^{\infty}(G)$, we define the function $\check{\varphi} \in L^{\infty}(G \times G)$ by $\check{\varphi}(g,h)=\varphi(g^{-1}h)$. The notation $\check{\varphi}$ will be used without further mentioning. In what follows, we will consider continuous functions $\varphi:G \lra \bbC$ such that $\check{\varphi}$ is a (completely bounded) Schur multiplier on $S^p(L^2(G))$.

\subsection{$KAK$ decomposition for Lie groups}
Recall that every connected semisimple Lie group $G$ with finite center can be decomposed as $G=KAK$, where $K$ is a maximal compact subgroup (unique up to conjugation) and $A$ is an abelian Lie group such that its Lie algebra $\mathfrak{a}$ is a Cartan subspace of the Lie algebra $\mathfrak{g}$ of $G$. The dimension of $\mathfrak{a}$ is called the real rank of $G$ and is denoted by $\rr(G)$. The $KAK$ decomposition is in general not unique. However, after choosing a set of positive roots and restricting to the closure $\overline{A^{+}}$ of the positive Weyl chamber $A^{+}$, we still have $G=K\overline{A^{+}}K$. Moreover, if $g=k_1ak_2$, where $k_1,k_2 \in K$ and $a \in \overline{A^{+}}$, then $a$ is unique. For more details, see \cite{helgasonlie}, \cite{knapp}.

\subsection{Gelfand pairs and spherical functions} \label{subsec:gpsf1}
Let $G$ be a Lie group with compact subgroup $K$. We denote the (left) Haar measure on $G$ by $dx$ and the normalized Haar measure on $K$ by $dk$. A function $\varphi:G \lra \bbC$ is said to be $K$-bi-invariant if $\varphi(k_1gk_2)=\varphi(g)$ for all $g \in G$ and $k_1,k_2 \in K$. Note that for $\varphi \in C(G)$, the continuous function defined by $\varphi^{K}(g)=\int_K \int_K \varphi(kgk^{\prime})dkdk^{\prime}$ is $K$-bi-invariant. By abuse of notation, we denote the space of $K$-bi-invariant compactly supported continuous functions on $G$ by $C_c(K \backslash G \slash K)$. This space can be considered as a subalgebra of the convolution algebra $C_c(G)$. If this subalgebra is commutative, then the pair $(G,K)$ is said to be a Gelfand pair. Equivalently, if $G$ is a Lie group with compact subgroup $K$, then $(G,K)$ is a Gelfand pair if and only if for every irreducible unitary representation $\pi$ of $G$ on a Hilbert space $\mathcal{H}_{\pi}$, the space $\mathcal{H}_{\pi_e}$ consisting of $K$-invariant vectors, i.e., $\mathcal{H}_{\pi_e}=\{ \xi \in \mathcal{H} \mid \forall k \in K:\,\pi(k)\xi=\xi \}$, is at most one-dimensional. Also, the pair $(G,K)$ is a Gelfand pair if and only if the representation $L^2(G \slash K)$ is multiplicity free.

Let $(G,K)$ be a Gelfand pair. A function $h \in C(K \backslash G \slash K)$ is called a spherical function if the functional $\chi$ on $C_c(K \backslash G \slash K)$ given by $\chi(\varphi)=\int_G \varphi(x)h(x^{-1})dx$ defines a nontrivial character, i.e., $\chi(\varphi \ast \psi)=\chi(\varphi)\chi(\psi)$ for all $\varphi,\psi \in C_c(K \backslash G \slash K)$. Spherical functions arise as the matrix coefficients of $K$-invariant vectors in irreducible representations of $G$.

It is possible to consider Gelfand pairs in more general settings than Lie groups, e.g., in the setting of locally compact groups (see \cite{vandijk},\cite{faraut}).

\subsection{Schur multipliers on compact Gelfand pairs}
Let $G$ and $K$ be Lie groups such that $(G,K)$ is a Gelfand pair, and let $X=G \slash K$ denote the homogeneous space (with quotient topology) corresponding with the canonical (transitive) action of $G$. It follows that $K$ is the stabilizer subgroup of a certain element $e_0 \in X$. In this section we consider Schur multipliers on the Schatten classes $S^p(\mathcal{H})$, where $\mathcal{H}=L^2(G)$ or $L^2(X)$. To this end, it is natural to look at multipliers on $G$ that are $K$-bi-invariant. Denote by $D$ the space $K \backslash G \slash K$ as a topological space, and denote by $f:K \backslash G \slash K \lra D$, $KgK \mapsto \xi$ the corresponding homeomorphism. It follows that every function $\varphi$ in $C(K \backslash G \slash K)$ induces a continuous function $\varphi^0$ on $D$ such that $\varphi(g)=\varphi^0(\xi)$ for all $g \in G$, where $\xi$ is the image under the homeomorphism $f$.

A Gelfand pair $(G,K)$ is called compact if $G$ is a compact group. In this section, all Gelfand pairs are assumed to be compact, unless explicitly stated otherwise. For compact groups every representation on a Hilbert space is equivalent to a unitary representation, every irreducible representation is finite-dimensional, and every unitary representation is the direct sum of irreducible ones. For an irreducible unitary representation $\pi$ of $G$ on a Hilbert space $\mathcal{H}_{\pi}$, let $P_{\pi}=\int_K \pi(k)dk$ denote the projection onto $\mathcal{H}_{\pi_e}$ (see Section \ref{subsec:gpsf1}), and let $\hat{G}_K$ denote the space of equivalence classes of the irreducible unitary representations $\pi$ of $G$ such that $P_{\pi} \neq 0$.
\begin{lem} \label{lem:peterweylspecial}
  Let $(G,K)$ be a compact Gelfand pair, and let $X=G \slash K$ be the corresponding (compact) homogeneous space. Then
\[
  L^2(X)=\oplus_{\pi \in \hat{G}_K} \mathcal{H}_{\pi}.
\]
Let $h_{\pi}$ denote the spherical function corresponding to the equivalence class $\pi$ of representations. Then for every $\varphi \in L^2(K \backslash G \slash K)$ we have
\[
  \varphi=\sum_{\pi \in \hat{G}_K} c_{\pi} \dim{\mathcal{H}_{\pi}} h_{\pi},
\]
where $c_{\pi}=\langle \varphi,h_{\pi} \rangle$. Moreover, denoting by $h^0_{\pi}$ the (spherical) function on $D$ corresponding to $h_{\pi}$, we have $\varphi^0=\sum_{\pi \in \hat{G}_K} c_{\pi} (\dim{\mathcal{H}_{\pi}})h^0_{\pi}$.
\end{lem}
This lemma follows from the Peter-Weyl theorem applied to a compact homogeneous space (see, e.g., \cite[Section V.4]{helgasongroupsanalysis}). The decomposition of $\varphi$ (and hence $\varphi^0$) is stated explicitly in \cite[Proposition 9.10.4]{wolf}.
\begin{lem} \label{lem:psiX}
  Let $(G,K)$ be a (not necessarily compact) Gelfand pair, and let $X=G \slash K$ denote the corresponding homogeneous space. Choose $e_0 \in X$ so that $K$ is its stabilizer subgroup. Let $\varphi \in C(K \backslash G \slash K)$. Then there exists a continuous function $\psi:X \times X \lra \bbC$ such that for all $g,h \in G$,
\[
  \varphi(g^{-1}h)=\psi(ge_0,he_0).
\]
\end{lem}
\begin{proof}
  If $ge_0=g^{\prime}e_0$ for $g,g^{\prime} \in G$, then $g^{-1}g^{\prime} \in K$, and hence $g^{\prime}=gk$ for some $k \in K$. Hence, by the $K$-bi-invariance of $\varphi$, we know that $\varphi(g^{-1}h)$ depends only on the pair $(ge_0,he_0) \in X \times X$, so there exists a function $\psi:X \times X \lra \bbC$ such that $\varphi(g^{-1}h)=\psi(ge_0,he_0)$. Since $X=G \slash K$ is equipped with the quotient topology, this function is continuous.
\end{proof}
\begin{lem} \label{lem:quotientspace}
Let $(G,K)$ be a compact Gelfand pair. If $\varphi:G \lra \bbC$ is a continuous $K$-bi-invariant function such that $\check{\varphi} \in cbMS^p(L^2(G))$ (see Section \ref{subsec:mloccpt}) for some $p \in [1,\infty]$, then $\|\psi\|_{cbMS^p(L^2(X))} = \|\check{\varphi}\|_{cbMS^p(L^2(G))}$, where $\psi:X \times X \lra \bbC$ is as defined in Lemma \ref{lem:psiX}. If $K$ is an infinite group, then these norms are equal to $\|\check{\varphi}\|_{MS^p(L^2(G))}$.
\end{lem}
\begin{proof}
By \cite[Lemma 1.1]{mackey}, the quotient map $G \lra G \slash K$ has a Borel cross section. Let $Y$ denote the image of this cross section. The result now follows directly from Lemma \ref{lem:schurbcb} by putting $\Omega=K$, so that $G=Y \times K$ as a measure space by the map $(y,k) \mapsto yk$ for $y \in Y$ and $k \in K$.
\end{proof}
We can now prove a decomposition result for Schur multipliers on $S^p(L^2(G))$ coming from $K$-bi-invariant functions.
\begin{prp} \label{prp:multiplierdecompositionestimate}
Let $(G,K)$ be a compact Gelfand pair, suppose that $K$ has infinitely many elements, and let $p \in [1,\infty)$. Let $\varphi:G \lra \bbC$ be a continuous $K$-bi-invariant function such that $\check{\varphi} \in MS^p(L^2(G))$. Then
\[
  \left( \sum_{\pi \in \hat{G}_K} |c_{\pi}|^p (\dim{\mathcal{H}_{\pi}}) \right)^{\frac{1}{p}} \leq \|\check{\varphi}\|_{MS^p(L^2(G))},
\]
where $c_{\pi}$ and $\mathcal{H}_{\pi}$ are as in Lemma \ref{lem:peterweylspecial}.
\end{prp}
\begin{proof}
As before, let $(T_kf)(x)=\int_G k(x,y)f(y)dy$. Then $T_1$ is the projection on $\bbC 1 \in L^2(X)$. It follows that $\|T_1\|_{S^p(L^2(X))}=1$. It is sufficient to prove that $(\sum_{\pi \in \hat{G}_K} |c_{\pi}|^p (\dim{\mathcal{H}_{\pi}}))^{\frac{1}{p}} \leq \|T_{\psi}\|_{S^p(L^2(X))}$, where $\psi$ is as before. Indeed, we have $\|T_{\psi}\|_{S^p(L^2(X))} = \frac{\|T_{\psi}\|_{S^p(L^2(X))}}{\|T_1\|_{S^p(L^2(X))}} \leq \|\psi\|_{MS^p(L^2(X))}$, which is smaller than or equal to $\|\psi\|_{cbMS^p(L^2(X))}=\|\check{\varphi}\|_{MS^p(L^2(G))}$ by Lemma \ref{lem:quotientspace} under the assumption that $K$ is an infinite group.

By Lemma \ref{lem:peterweylspecial}, we have $\varphi=\sum_{\pi \in \hat{G}_K} c_{\pi} \dim{\mathcal{H}_{\pi}} h_{\pi}$. By \cite[Theorem V.4.3]{helgasongroupsanalysis}, it follows that the operator $P_{\mathcal{H}_{\pi}}=\dim{\mathcal{H}}_{\pi}T_{h^{\prime}_{\pi}}$ is the projection onto $\mathcal{H}_{\pi}$, where $h^{\prime}_{\pi}:X \times X \lra \bbC$ denotes the function induced by $h_{\pi}$ (see Lemma \ref{lem:psiX}). Since $L^2(X)$ decomposes as a direct sum of Hilbert spaces, we have
\begin{equation} \nonumber
\begin{split}
  \left\|T_{\psi}\right\|_{S^p(L^2(X))}^p &= \left\|\sum_{\pi \in \hat{G}_K} c_{\pi} \dim{\mathcal{H}_{\pi}} T_{h^{\prime}_{\pi}}\right\|_{S^p(L^2(X))}^p \\
    &= \sum_{\pi \in \hat{G}_K} |c_{\pi}|^p \Tr(|P_{\mathcal{H}_{\pi}}|^p)=\sum_{\pi \in \hat{G}_K} |c_{\pi}|^p \dim{\mathcal{H}_{\pi}}.
\end{split}
\end{equation}
\end{proof}
\begin{lem} \label{lem:averaging}
  Let $G$ be a locally compact group with compact subgroup $K$. For $p \in [1,\infty]$, let $\varphi \in C(G)$ be such that $\check{\varphi} \in MS^p(L^2(G))$. Then the continuous function $\varphi^K$ defined by $\varphi^K(g)=\int_K\int_K \varphi(kgk^{\prime})dkdk^{\prime}$ induces an element $\check{\varphi}^K$ of $MS^p(L^2(G))$, and $\|\check{\varphi}^K\|_{MS^p(L^2(G))} \leq \|\check{\varphi}\|_{MS^p(L^2(G))}$. The analogous statement holds in the completely bounded case.
\end{lem}
\begin{proof}
  Let $\nu_n$ be a sequence of finitely supported probability measures on $K$ pointwise converging to the Haar measure $\mu$. Let $\varphi_n:G \lra \bbC$ be defined by $\varphi_n(g)=\int_K\int_K \varphi(kgk^{\prime})d\nu_n(k)d\nu_n(k^{\prime})$. Each $\varphi_n$ is a convex combination of functions $_{k}\varphi_{k^{\prime}}$ of the form $_{k}\varphi_{k^{\prime}}(g)=\varphi(kgk^{\prime})$, where $k,k^{\prime} \in K$ are fixed. Hence, $\varphi^K$ is an element of the pointwise closure of $\conv\{_{k}\varphi_{k^{\prime}} \mid k,k^{\prime} \in K\}$. One easily checks that for all $k,k^{\prime} \in K$, we have $\|_{k}\check{\varphi}_{k^{\prime}}\|_{MS^p(L^2(G))} = \|\check{\varphi}\|_{MS^p(L^2(G))}$. Hence, by Lemma \ref{lem:finitecharacterization}, we have $\check{\varphi}^K \in MS^p(L^2(G))$, and $\|\check{\varphi}^K\|_{MS^p(L^2(G))} \leq \|\check{\varphi}\|_{MS^p(L^2(G))}$. The result for the completely bounded case follows in an analogous way.
\end{proof}

\subsection{The property $\apschur$} \label{sec:apschur}
In this section we recall the definition of the $\apschur$, as given by Lafforgue and de la Salle in \cite{ldls}. First, recall that the Fourier algebra $A(G)$ (see \cite{eymard}) consists of the coefficients of the left-regular representation of $G$. More precisely, $\varphi \in A(G)$ if and only if there exist $\xi,\eta \in L^2(G)$ such that for all $x \in G$ we have $\varphi(x)=\langle \lambda(x)\xi,\eta \rangle$. With the norm $\|\varphi\|_{A(G)}=\min \{ \|\xi\|\|\eta\| \mid \forall x \in G \; \varphi(x)=\langle \lambda(x)\xi,\eta \rangle \}$, it is a Banach space.
\begin{dfn}(\cite[Definition 2.2]{ldls}) Let $G$ be a locally compact Hausdorff second countable group, and let $1 \leq p \leq \infty$. The group $G$ is said to have the property of completely bounded approximation by Schur multipliers on $S^p$, denoted $\apschur$, if there exists a constant $C > 0$ and a net $\varphi_{\alpha} \in A(G)$ such that $\varphi_{\alpha} \to 1$ uniformly on compacta and $\sup_{\alpha} \|\check{\varphi}_{\alpha}\|_{cbMS^p(L^2(G))} \leq C$. The infimum of these $C$'s is denoted by $\lambdaapschur (G)$.
\end{dfn}
The following result is a key property of the $\apschur$ (see \cite[Theorem 2.5]{ldls}).
\begin{thm} \label{thm:apschurlattices}
  Let $G$ be a locally compact Hausdorff group, and let $\Gamma$ be a lattice in $G$. Then for $1 \leq p \leq \infty$, we have $\Lambda_{p,\mathrm{cb}}^{\mathrm{Schur}}(\Gamma)=\Lambda_{p,\mathrm{cb}}^{\mathrm{Schur}}(G)$.
\end{thm}
Lafforgue and de la Salle also proved that for a discrete group $\Gamma$ and $p \in (1,\infty)$, it follows that $\Lambda_{p,\mathrm{cb}}^{\mathrm{Schur}}(\Gamma) \in \{1,\infty\}$. Since a semisimple Lie group $G$ has lattices \cite{borelharishchandra}, we conclude by the above proposition that for such a group, it also follows that $\Lambda_{p,\mathrm{cb}}^{\mathrm{Schur}}(G) \in \{1,\infty\}$ for $p \in (1,\infty)$.
\begin{prp} \label{prp:apschurproperties}
Let $G$ be a locally compact Hausdorff group. The $\apschur$ satisfies the following properties:
\begin{enumerate}
 \item for $p=\infty$ (or $p=1$, by the third statement of this proposition), the group $G$ has the $\apschur$ if and only if it is weakly amenable, and $\lambdaapschur (G)=\Lambda(G)$, where $\Lambda(G)$ denotes the Cowling-Haagerup constant of $G$;
 \item for every locally compact group, $\Lambda_{2,\mathrm{cb}}^{\mathrm{Schur}}(G)=1$;
 \item if $p,q \in [1,\infty]$ such that $\frac{1}{p}+\frac{1}{q}=1$, then $\Lambda_{p,\mathrm{cb}}^{\mathrm{Schur}}(G)=\Lambda_{q,\mathrm{cb}}^{\mathrm{Schur}}(G)$;
 \item if $2 \leq p \leq q \leq \infty$, then $\Lambda_{p,\mathrm{cb}}^{\mathrm{Schur}}(G) \leq \Lambda_{q,\mathrm{cb}}^{\mathrm{Schur}}(G)$;
 \item if $H$ is a closed subgroup of $G$ and $1 \leq p \leq \infty$, then $\Lambda_{p,\mathrm{cb}}^{\mathrm{Schur}}(H) \leq \Lambda_{p,\mathrm{cb}}^{\mathrm{Schur}}(G)$;
 \item if $G$ has a compact subgroup $K$, and if $\varphi_{\alpha}$ is a net in $A(G)$ converging to $1$ uniformly on compacta such that $\sup_{\alpha} \|\check{\varphi}_{\alpha}\|_{cbMS^p(L^2(G))} \leq C$, then there exists a net $\tilde{\varphi}_{\alpha}$ in $A(G) \cap C(K \backslash G \slash K)$ such that $\sup_{\alpha} \|\check{\tilde{\varphi}}_{\alpha}\|_{cbMS^p(L^2(G))} \leq C$ that converges to $1$ uniformly on compacta.
 \item if $K$ is a compact normal subgroup of $G$ and $1 \leq p \leq \infty$, then $\Lambda_{p,\mathrm{cb}}^{\mathrm{Schur}}(G)=\Lambda_{p,\mathrm{cb}}^{\mathrm{Schur}}(G \slash K)$;
 \item if $G_1$ and $G_2$ are locally isomorphic connected (semi)simple Lie groups with finite centers, then for $p \in [1,\infty]$, we have $\Lambda_{p,\mathrm{cb}}^{\mathrm{Schur}}(G_1)=\Lambda_{p,\mathrm{cb}}^{\mathrm{Schur}}(G_2)$;
\end{enumerate}
\end{prp}
\begin{proof}
  The first statement is clear. The second through the fifth statement are covered in \cite[Section 2]{ldls}. The sixth statement follows from Lemma \ref{lem:averaging}. By combining the sixth statement and Lemma \ref{lem:quotientspace}, the seventh statement follows. The fact that the net on the group converges uniformly on compacta if and only if the net on the quotient does, is straightforward (see \cite{cowlinghaagerup}). For the eighth statement, note that the center is a normal subgroup of a group. Using the seventh statement and the fact that the adjoint groups $G_1 \slash Z(G_1)$ and $G_2 \slash Z(G_2)$, where $Z(G_i)$ denotes the center of $G_i$, are isomorphic, we obtain the result.
\end{proof}

\subsection{Approximation properties for noncommutative $L^p$-spaces}
The operator space structure on a noncommutative $L^p$-space $L^p(M,\tau)$ can be obtained by considering this space as a certain interpolation space (see \cite{kosaki}). Indeed, the pair of spaces $(M,L^1(M,\tau))$ becomes a compatible couple of operator spaces, and for $1 < p < \infty$ we have the isometry $L^p(M,\tau) \cong [M,L^1(M,\tau)]_{\frac{1}{p}}$. By \cite[Lemma 1.7]{pisiernoncomvvlpsp}, we know that for a linear map $T:L^p(M,\tau) \lra L^p(M,\tau)$, its completely bounded norm $\|T\|_{cb}$ corresponds to $\sup_{n \in \mathbb{N}} \|\id_{S_n^p} \otimes\; T:S^p_n[L^p(M)] \lra S^p_n[L^p(M)]\|$. Using \cite[Corollary 1.4]{pisiernoncomvvlpsp} and the fact that $S^1_n \otimes L^1(M)=L^1(M \otimes M_n)$, we obtain that $S^p_n[L^p(M)]=L^p(M \otimes M_n)$, which implies that $\|T\|_{cb}=\sup_{n \in \mathbb{N}}\|T \otimes \id:L^p(M \otimes M_n) \lra L^p(M \otimes M_n)\|$.

In Section \ref{sec:introduction} of this article, we recalled the definition of the CBAP, CCAP and OAP. It was shown by Junge and Ruan \cite{jungeruan} that if $\Gamma$ is a discrete group with the AP (of Haagerup and Kraus), and if $p \in (1,\infty)$, then $L^p(L(\Gamma))$ has the OAP, where $L(\Gamma)$ denotes the group von Neumann algebra of $\Gamma$. Lafforgue and de la Salle related the AP for groups and the OAP for noncommutative $L^p$-spaces to the $\apschur$.
\begin{lem} (\cite[Corollary 3.12]{ldls}) \label{lem:apapschur}
  If $\Gamma$ is a countable discrete group with the AP, and if $p \in (1,\infty)$, then $\Lambda_{p,\mathrm{cb}}^{\mathrm{Schur}}(\Gamma)=1$.
\end{lem}
\begin{lem} (\cite[Corollary 3.13]{ldls}) \label{lem:oapapschur}
  If $p \in (1,\infty)$ and $\Gamma$ is a countable discrete group such that $L^p(L(\Gamma))$ has the OAP, then $\Lambda_{p,\mathrm{cb}}^{\mathrm{Schur}}(\Gamma)=1$.
\end{lem}
One of the main results of Lafforgue and de la Salle is the following.
\begin{thm}(\cite[Theorem E]{ldls}) \label{thm:ldlsmain}
  Let $n \geq 3$. For $p \in [1,\frac{4}{3}) \cup (4,\infty]$, the (exact) group $\SL(n,\bbR)$ does not have the $\apschur$.
\end{thm}
As a consequence, the group $\SL(n,\bbR)$ does not have the AP, and for $p \in [1,\frac{4}{3}) \cup (4,\infty]$ and a lattice $\Gamma$ in $\SL(n,\bbR)$, the noncommutative $L^p$-space $L^p(L(\Gamma))$ does not have the OAP or CBAP.

\section{The group $\Sp(2,\bbR)$} \label{sec:sp2}
In this section, we prove the following theorem. The proof is along the same lines as the proof of the failure of the AP for $\Sp(2,\bbR)$ in \cite{haagerupdelaat1} (and for some details we will refer to that article), but obtaining sufficiently sharp estimates for Schur multipliers on Schatten classes is technically more involved.
\begin{thm} \label{thm:sp2notapschur}
  For $p \in [1,\frac{12}{11}) \cup (12,\infty]$, the group $\Sp(2,\bbR)$ does not have the $\apschur$.
\end{thm}
In this section, we write $G=\Sp(2,\bbR)$. Recall that $G$ is defined as the Lie group
\[
	G:=\{g \in \GL(4,\bbR) \mid g^t J g = J\},
\]
where
\[
  J=\left( \begin{array}{cc} 0 & I_2 \\ -I_2 & 0 \end{array} \right).
\]
Here $I_2$ denotes the $2 \times 2$ identity matrix. The maximal compact subgroup $K$ of $G$ is isomorphic to $\U(2)$ and explicitly given by
\[
  K= \bigg\{ \left( \begin{array}{cc} A & -B \\ B & A \end{array} \right) \in \mathrm{M}_4(\bbR) \biggm\vert A+iB \in \U(2) \bigg\}.
\]
Let $\overline{A^{+}}=\left\{D(\alpha_1,\alpha_2)= \diag(e^{\alpha_1},e^{\alpha_2},e^{-\alpha_1},e^{-\alpha_2}) \mid \alpha_1 \geq \alpha_2 \geq 0\right\}$. It follows that $G=K\overline{A^{+}}K$.

For $p=1$ and $\infty$, the $\apschur$ is equivalent to weak amenability (as mentioned in Proposition \ref{prp:apschurproperties}), and the failure of weak amenability for $G$ was proved in \cite{haagerupgroupcsacbap}. Therefore, we can restrict ourselves to the case $p \in (1,\infty)$. As follows from Proposition \ref{prp:apschurproperties}, it suffices to consider approximating nets consisting of $K$-bi-invariant functions. The following result gives a certain asymptotic behaviour of continuous $K$-bi-invariant functions $\varphi$ for which the induced function $\check{\varphi}$ is a Schur multiplier on $S^p(L^2(G))$. From this, it follows that the constant function $1$ cannot be approximated pointwise (and hence not uniformly on compacta) by a $K$-bi-invariant net in $A(G)$ in such a way that the net of associated multipliers is uniformly bounded in the $MS^p(L^2(G))$-norm. This implies Theorem \ref{thm:sp2notapschur}.
\begin{prp} \label{prp:sp2abapschur}
  Let $p > 12$. There exist constants $C_1(p),C_2(p)$ (depending on $p$ only) such that for all $\varphi \in C(K \backslash G \slash K)$ for which $\check{\varphi} \in MS^p(L^2(G))$, the limit $\varphi_{\infty}=\lim_{\|\alpha\| \to \infty} \varphi(D(\alpha_1,\alpha_2))$ exists, and for all $\alpha_1 \geq \alpha_2 \geq 0$,
\[
  |\varphi(D(\alpha_1,\alpha_2))-\varphi_{\infty}| \leq C_1(p)\|\check{\varphi}\|_{MS^p(L^2(G))}e^{-C_2(p)\|a\|_2},
\]
where $\|\alpha\|_2=\sqrt{\alpha_1^2+\alpha_2^2}$.
\end{prp}
\begin{rmk}
  Note that Proposition \ref{prp:sp2abapschur} is stated in terms of the $MS^p(L^2(G))$-norm rather than the $cbMS^p(L^2(G))$-norm. However, we have $\|.\|_{MS^p(L^2(G))} \leq \|.\|_{cbMS^p(L^2(G))}$, which shows that Proposition \ref{prp:sp2abapschur} is indeed sufficient to prove Theorem \ref{thm:sp2notapschur}. Moreover, by \cite[Theorem 1.18]{ldls}, the claims are equivalent for non-discrete groups.
\end{rmk}
For the proof of Proposition \ref{prp:sp2abapschur}, we will identify two Gelfand pairs in $G$ and describe certain properties of their spherical functions.

Consider the group $\U(2)$, which contains the circle group $\U(1)$ as a subgroup via the embedding
\[
  \U(1) \hookrightarrow \left( \begin{array}{cc} 1 &  0 \\ 0 & \U(1) \end{array} \right) \subset \U(2).
\]
Let $K_1$ denote the copy of $\U(1)$ in $G$ under the identification of $\U(2)$ with $K$. It goes back to Weyl \cite{weyl} that $(\U(2),\U(1))$ is a Gelfand pair (see, e.g., \cite[Theorem IX.9.14]{knapp}). The homogeneous space $\U(2) \slash \U(1)$ is homeomorphic to the complex $1$-sphere $S_{\bbC}^1 \subset \bbC^2$ and the double coset space $\U(1) \backslash \U(2) \slash \U(1)$ is homeomorphic to the closed unit disc $\overline{\mathbb{D}} \subset \bbC$ by the map
\[
  \U(1)\left( \begin{array}{cc} u_{11} & u_{12} \\ u_{21} & u_{22} \end{array} \right)\U(1) \mapsto u_{11}.
\]
The spherical functions for $(\U(2),\U(1))$ can be found in \cite{koornwinder}. By the homeomorphism $\U(1) \backslash \U(2) \slash \U(1) \cong \overline{\mathbb{D}}$, they can be considered as functions of one complex variable in the closed unit disc. They are indexed by the integers $l,m \geq 0$ and explicitly given by
\[
  h_{l,m} \left( \begin{array}{cc} u_{11} & u_{12} \\ u_{21} & u_{22} \end{array} \right)=h_{l,m}^0(u_{11}),
\]
where in the point $z \in \overline{\mathbb{D}}$, the function $h_{l,m}^0$ is explicitly given by
\[
  h_{l,m}^0(z)=\left\{ \begin{array}{ll} z^{l-m} P_m^{(0,l-m)}(2|z|^2-1) & \qquad l \geq m, \\ \overline{z}^{m-l} P_l^{(0,m-l)}(2|z|^2-1) & \qquad l < m. \end{array} \right.
\]
Here $P_n^{(\alpha,\beta)}$ denotes the $n^{\textrm{th}}$ Jacobi polynomial. These spherical functions satisfy a certain H\"older continuity condition, as is stated in the following lemma (see \cite[Corollary 3.5]{haagerupdelaat1}). The proof of this Lemma makes use of recent results by Haagerup and Schlichtkrull \cite{hsjacobi}.
\begin{lem} \label{lem:hoelderu2}
  For all $l,m \geq 0$, and for $\theta_1,\theta_2 \in [0,2\pi)$, we have
\[
  \biggl\vert h_{l,m}^0 \left( \frac{e^{i\theta_1}}{\sqrt{2}} \right) - h_{l,m}^0 \left( \frac{e^{i\theta_2}}{\sqrt{2}} \right) \biggr\vert \leq C(l+m+1)^{\frac{3}{4}}|\theta_1-\theta_2|,
\]
\[
  \biggl\vert h_{l,m}^0 \left( \frac{e^{i\theta_1}}{\sqrt{2}} \right) - h_{l,m}^0 \left( \frac{e^{i\theta_2}}{\sqrt{2}} \right) \biggr\vert \leq 2C(l+m+1)^{-\frac{1}{4}}.
\]
Here $C>0$ is a uniform constant. Combining the two, we get
\[
  \biggl\vert h_{l,m}^0 \left( \frac{e^{i\theta_1}}{\sqrt{2}} \right) - h_{l,m}^0 \left( \frac{e^{i\theta_2}}{\sqrt{2}} \right) \biggr\vert \leq 2^{\frac{3}{4}}C|\theta_1-\theta_2|^{\frac{1}{4}}.
\]
\end{lem}
Let $\varphi:\U(2) \lra \bbC$ be a $\U(1)$-bi-invariant continuous function. Then
\[
  \varphi(u)=\varphi\left( \begin{array}{ll} u_{11} & u_{12} \\ u_{21} & u_{22} \end{array} \right)=\varphi^0(u_{11}), \quad u \in \U(2),\,u_{11} \in \overline{\mathbb{D}},
\]
for some continuous function $\varphi^0:\overline{\mathbb{D}} \lra \bbC$. By Lemma \ref{lem:peterweylspecial}, we know that $L^2(X)=\oplus_{l,m \geq 0} \mathcal{H}_{l,m}$, where $X=\U(2) \slash \U(1) \cong S_1^{\bbC}$. It is known that $\dim{\mathcal{H}_{l,m}}=l+m+1$, so, by Proposition \ref{prp:multiplierdecompositionestimate}, we get
\[
  \varphi^0=\sum_{l,m=0}^{\infty} c_{l,m}(l+m+1) h_{l,m}^0,
\]
for certain $c_{l,m} \in \bbC$. Moreover, by the same proposition, we obtain that if $p \in (1,\infty)$, then $(\sum_{l,m \geq 0} |c_{l,m}|^p (l+m+1))^{\frac{1}{p}} \leq \|\check{\varphi}\|_{MS^p(L^2(\U(2)))}$, where $\check{\varphi}$ is defined as above by $\check{\varphi}(g,h)=\varphi(g^{-1}h)$.
\begin{lem} \label{lem:behavioru2}
Let $p > 12$, and let $\varphi:\U(2) \lra \bbC$ be a continuous $\U(1)$-bi-invariant function such that $\check{\varphi}$ is an element of $MS^p(L^2(\U(2)))$. Then $\varphi^0$ satisfies
\[
  \left|\varphi^0\left(\frac{e^{i\theta_1}}{\sqrt{2}}\right)-\varphi^0\left(\frac{e^{i\theta_2}}{\sqrt{2}}\right)\right| \leq \tilde{C}(p)\|{\check{\varphi}}\|_{MS^p(L^2(\U(2)))}|\theta_1-\theta_2|^{\frac{1}{8}-\frac{3}{2p}}
\]
for $\theta_1,\theta_2 \in [0,2\pi)$. Here, $\tilde{C}(p)$ is a constant depending only on $p$.
\end{lem}
\begin{proof}
Let $p,q \in (1,\infty)$ be such that $\frac{1}{p}+\frac{1}{q}=1$. Then for $\theta_1,\theta_2 \in [0,2\pi)$,
\begin{equation} \nonumber
\begin{split}
  &\left|\varphi^0\left(\frac{e^{i\theta_1}}{\sqrt{2}}\right)-\varphi^0\left(\frac{e^{i\theta_2}}{\sqrt{2}}\right)\right| = \sum_{l,m \geq 0} |c_{l,m}|(l+m+1)\left|h_{l,m}^0\left(\frac{e^{i\theta_1}}{\sqrt{2}}\right)-h_{l,m}^0\left(\frac{e^{i\theta_2}}{\sqrt{2}}\right)\right| \\
  &\leq \left( \sum_{l,m \geq 0} |c_{l,m}|^q(l+m+1) \right)^{\frac{1}{q}} \left( \sum_{l,m \geq 0} (l+m+1)\left|h_{l,m}^0\left(\frac{e^{i\theta_1}}{\sqrt{2}}\right)-h_{l,m}^0\left(\frac{e^{i\theta_2}}{\sqrt{2}}\right)\right|^p \right)^{\frac{1}{p}} \\
  &\leq \|\check{\varphi}\|_{MS^q(L^2(\U(2)))} \left( \sum_{l,m \geq 0} (l+m+1) \left|h_{l,m}^0\left(\frac{e^{i\theta_1}}{\sqrt{2}}\right)-h_{l,m}^0\left(\frac{e^{i\theta_2}}{\sqrt{2}}\right)\right|^p \right)^{\frac{1}{p}}.
\end{split}
\end{equation}
Note that $\|\check{\varphi}\|_{MS^q(L^2(\U(2)))}=\|\check{\varphi}\|_{MS^p(L^2(\U(2)))}$. If we look at the terms of the last sum, we get, using Lemma \ref{lem:hoelderu2} and the fact that $\min\{x,y\} \leq x^{\epsilon}y^{1-\epsilon}$ for $x,y > 0$ and $\epsilon \in (0,1)$, that
\begin{equation} \nonumber
\begin{split}
    &(l+m+1)  \left|h_{l,m}^0\left(\frac{e^{i\theta_1}}{\sqrt{2}}\right)-h_{l,m}^0\left(\frac{e^{i\theta_2}}{\sqrt{2}}\right)\right|^p \\
    &\quad\leq \min\{C^p(l+m+1)^{1+\frac{3}{4}p}|\theta_1-\theta_2|^p,2^pC^p(l+m+n)^{1-\frac{1}{4}p}\} \\
    &\quad\leq 2^{p(1-\epsilon)}C^p |\theta_1-\theta_2|^{p\epsilon} (l+m+1)^{1+p\epsilon-\frac{1}{4}p}
\end{split}
\end{equation}
for $\epsilon \in (0,1)$. Hence, the sum converges for $0 < \epsilon < \frac{1}{4}-\frac{3}{p}$. Such an $\epsilon$ only exists for $p > 12$. Hence, if $p > 12$, and putting $\epsilon=\frac{1}{2}(\frac{1}{4}-\frac{3}{p})=\frac{1}{8}-\frac{3}{2p}$, then
\[
  \left|\varphi^0\left(\frac{e^{i\theta_1}}{\sqrt{2}}\right)-\varphi^0\left(\frac{e^{i\theta_2}}{\sqrt{2}}\right)\right| \leq \tilde{C}(p) \|\check{\varphi}\|_{MS^p(L^2(\U(2))} |\theta_1-\theta_2|^{\frac{1}{8}-\frac{3}{2p}}
\]
for some constant $\tilde{C}(p)$ depending only on $p$.
\end{proof}
For $\alpha \in \bbR$ consider the map $K \lra G$ defined by $k \mapsto D_{\alpha}kD_{\alpha}$, where $D_{\alpha} = \diag(e^{\alpha},1,e^{-\alpha},1)$.
\begin{lem} \label{lem:fromGtoK}
	Let $\varphi:G \lra \bbC$ be a continuous $K$-bi-invariant function such that $\check{\varphi} \in MS^p(L^2(G))$ for some $p \in (1,\infty)$, and for $\alpha \in \bbR$, let $\psi_{\alpha}:K \lra \bbC$ be defined by $\psi_{\alpha}(k)=\varphi(D_{\alpha}kD_{\alpha})$. Then $\psi_{\alpha}$ is $K_1$-bi-invariant and satisfies
	\[
		\|\check{\psi}_{\alpha}\|_{MS^p(L^2(\U(2))} \leq \|\check{\varphi}\|_{MS^p(L^2(G))}.
	\]
\end{lem}
\begin{proof}
Using the fact that the group elements $D_{\alpha}$ commute with $K_1$, it follows that for all $k \in K$ and $k_1,k_2 \in K_1 \subset K_2$,
\[
  \psi_{\alpha}(k_1kk_2)=\varphi(D_{\alpha}k_1kk_2D_{\alpha})=\varphi(k_1D_{\alpha}kD_{\alpha}k_2)=\varphi(D_{\alpha}kD_{\alpha})=\psi_{\alpha}(k),
\]
so $\psi_{\alpha}$ is $K_1$-bi-invariant.

The second part follows by the fact that $D_{\alpha}KD_{\alpha}$ is a subset of $G$ and by applying Lemma \ref{lem:finitecharacterization}.
\end{proof}
From the fact that $\psi_{\alpha}$ is $K_1$-bi-invariant, it follows that $\psi_{\alpha}(u)=\psi_{\alpha}^0(u_{11})$, where $\psi_{\alpha}^0:\overline{\mathbb{D}} \lra \bbC$ is a continuous function.

Suppose that $\alpha_1 \geq \alpha_2 \geq 0$, and let $D(\alpha_1,\alpha_2)$ be as defined above. If we find an element of the form $D_{\alpha}kD_{\alpha}$ in $KD(\alpha_1,\alpha_2)K$, we can relate the value of a $K$-bi-invariant multiplier $\varphi$ to the value of the multiplier $\psi_{\alpha}$ that was just defined. This only works for certain $\alpha_1,\alpha_2 \geq 0$. It turns out to be sufficient to consider certain candidates for $k$, namely the ones of the form
\begin{equation} \label{eq:uform}
u=\left( \begin{array}{cc} a+ib & -\sqrt{1-a^2-b^2} \\ \sqrt{1-a^2-b^2} & a-ib \end{array} \right)
\end{equation}
with $a^2+b^2 \leq 1$. For a proof of the following result, see \cite[Lemma 3.9]{haagerupdelaat1}.
\begin{lem} \label{lem:hyperbolaseqs}
  Let $\alpha \geq 0$ and $\beta \geq \gamma \geq 0$. If $u \in K$ is of the form \eqref{eq:uform} with respect to the identification of $K$ with $\U(2)$, then $D_{\alpha}uD_{\alpha} \in KD(\beta,\gamma)K$ if and only if
\begin{equation} \label{eq:hyperbolaseqs}
  \begin{cases}
    & \sinh \beta \sinh \gamma = \sinh^2 \alpha (1-a^2-b^2), \\  
    & \sinh \beta - \sinh \gamma = \sinh(2\alpha)|a|.
  \end{cases}
  \end{equation}
\end{lem}
Consider the second Gelfand pair sitting inside $G$, namely the pair of groups $(\SU(2),\SO(2))$. Both groups are subgroups of $\U(2)$, so under the embedding into $G$, they give rise to compact Lie subgroups of $G$. The subgroup corresponding to $\SU(2)$ will be called $K_2$, and the one corresponding to $\SO(2)$ will be called $K_3$. The group $K_3$ commutes with the group generated by the elements $D_{\alpha}^{\prime}=\diag(e^{\alpha},e^{\alpha},e^{-\alpha},e^{-\alpha})$, where $\alpha \in \bbR$. The subgroup $\SU(2) \subset \U(2)$ consists of matrices of the form
\begin{equation} \nonumber
  u=\left( \begin{array}{cc} a+ib & -c+id \\ c+id & a-ib \end{array} \right),
\end{equation}
with $a,b,c,d \in \bbR$ such that $a^2+b^2+c^2+d^2=1$.

By \cite[Theorem 47.6]{bump}, the pair $(\SU(2),\SO(2))$ is a Gelfand pair. This also follows from \cite[Chapter 9]{farautanalysisonliegroups}. The homogeneous space $\SU(2)\slash \SO(2)$ is the sphere $S^2$, and the spherical functions on the double coset space $[-1,1]$ are indexed by $n \geq 0$, and given by the Legendre polynomials
\[
  P_n(2(a^2+c^2)-1)=P_n(a^2-b^2+c^2-d^2).
\]
Note that the double cosets of $\SO(2)$ in $\SU(2)$ are labeled by $a^2-b^2+c^2-d^2$. We use the following estimate (see \cite[Lemma 3.11]{haagerupdelaat1}).
\begin{lem} \label{lem:hoeldersu2}
For all non-negative integers $n$, and $x,y \in [-\frac{1}{2},\frac{1}{2}]$,
\begin{equation} \nonumber
\begin{split}
  |P_n(x)-P_n(y)| &\leq |P_n(x)|+|P_n(y)| \leq \frac{4}{\sqrt{n}},\\
  |P_n(x)-P_n(y)| &\leq \left|\int_x^y P_n^{\prime}(t)dt\right| \leq 4\sqrt{n}|x-y|.
\end{split}
\end{equation}
Combining the two, we get
	\[
	  |P_n(x)-P_n(y)|\leq 4|x-y|^{\frac{1}{2}}
	\]
for $x,y \in [-\frac{1}{2},\frac{1}{2}]$, i.e., the Legendre polynomials are uniformly H\"older continuous on $[-\frac{1}{2},\frac{1}{2}]$ with exponent $\frac{1}{2}$.
\end{lem}
Let $\varphi:\SU(2) \lra \bbC$ be a $\SO(2)$-bi-invariant continuous function. Then
\[
  \varphi(u)=\varphi\left( \begin{array}{ll} a+ib & -c+id \\ c+id & a-ib \end{array} \right)=\varphi^0(2(a^2+c^2)-1)=\varphi^0(a^2-b^2+c^2-d^2),
\]
where $u \in \U(2),\,u_{11} \in \overline{\mathbb{D}}$, and where $\varphi^0:\overline{\mathbb{D}} \lra \bbC$ is some continuous function. By Lemma \ref{lem:peterweylspecial}, we know that $L^2(X)=\oplus_{n \geq 0} \mathcal{H}_n$, where $X = \SU(2) \slash \SO(2) \cong S^2$. It is known that $\dim{\mathcal{H}_n}=2n+1$, so, by Proposition \ref{prp:multiplierdecompositionestimate}, we get
\[
  \varphi^0=\sum_{n=0}^{\infty} c_n(2n+1) P_n,
\]
for certain $c_n \in \bbC$. Moreover, by the same proposition, we obtain that if $p \in (1,\infty)$, then $(\sum_{n \geq 0} |c_n|^p (2n+1))^{\frac{1}{p}} \leq \|\check{\varphi}\|_{MS^p(L^2(\SU(2)))}$, where $\check{\varphi}$ is defined as above by $\check{\varphi}(g,h)=\varphi(g^{-1}h)$.
\begin{lem} \label{lem:behaviorsu2}
Let $p>4$, and let $\varphi:\SU(2) \lra \bbC$ be a continuous $\SO(2)$-bi-invariant function such that $\check{\varphi} \in MS^p(L^2(\SU(2)))$. Then $\varphi^0$ satisfies
\[
  |\varphi^0(\delta_1)-\varphi^0(\delta_2)| \leq \hat{C}(p)\|{\varphi}\|_{MS^p(L^2(\SU(2))}|\delta_1-\delta_2|^{\frac{1}{4}-\frac{1}{p}}
\]
for $\delta_1,\delta_2 \in [-\frac{1}{2},\frac{1}{2}]$. Here $\hat{C}(p)$ is a constant depending only on $p$.
\end{lem}
\begin{proof}
Let $p,q \in (1,\infty)$ be such that $\frac{1}{p}+\frac{1}{q}=1$, and let $\delta_1,\delta_2 \in [-\frac{1}{2},\frac{1}{2}]$. Then
\begin{equation} \nonumber
\begin{split}
  &|\varphi^0(\delta_1)-\varphi^0(\delta_2)|=\sum_{n \geq 0} |c_n|(2n+1)|P_n(\delta_1)-P_n(\delta_2)| \\
    &\quad\leq \left( \sum_{n \geq 0} |c_n|^q(2n+1) \right)^{\frac{1}{q}} \left( \sum_{n \geq 0} (2n+1)|P_n(\delta_1)-P_n(\delta_2)|^p \right)^{\frac{1}{p}} \\
    &\quad\leq \|\check{\varphi}\|_{MS^q(L^2(\SU(2)))} \left( \sum_{n \geq 0} (2n+1) |P_n(\delta_1)-P_n(\delta_2)|^p \right)^{\frac{1}{p}}.
\end{split}
\end{equation}
Note that $\|\check{\varphi}\|_{MS^q(L^2(\SU(2)))}=\|\check{\varphi}\|_{MS^p(L^2(\SU(2)))}$. If we look at the terms of the last sum, we get, using Lemma \ref{lem:hoeldersu2} and the fact that $\min\{x,y\} \leq x^{\epsilon}y^{1-\epsilon}$ for $x,y > 0$ and $\epsilon \in (0,1)$, that
\begin{equation} \nonumber
\begin{split}
  (2n+1) |P_n(\delta_1)-P_n(\delta_2)|^p &\leq \min \{4^p(2n+1) n^{-\frac{p}{2}}, 4^p(2n+1) n^{\frac{p}{2}} |\delta_1-\delta_2|^p \} \\
    &\leq 4^p (3n)^{1+p\epsilon-\frac{p}{2}} |\delta_1-\delta_2|^{p\epsilon}
\end{split}
\end{equation}
for $\epsilon \in (0,1)$. Hence, the sum converges for $\epsilon \in (0,\frac{1}{2}-\frac{2}{p})$. Such an $\epsilon$ only exists for $p > 4$. Hence, if $p > 4$, and putting $\epsilon=\frac{1}{2}(\frac{1}{2}-\frac{2}{p})=\frac{1}{4}-\frac{1}{p}$, we have
\[
  |\varphi^0(\delta_1)-\varphi^0(\delta_2)| \leq \hat{C}(p)\|\check{\varphi}\|_{MS^pL^2(\U(2))}|\delta_1-\delta_2|^{\frac{1}{4}-\frac{1}{p}},
\]
where $\hat{C}(p)$ is a constant depending only on $p$.
\end{proof}
For $\alpha \in \bbR$ consider the map $K \lra G$ defined by $k \mapsto D_{\alpha}^{\prime}kvD_{\alpha}^{\prime}$, where $D_{\alpha}^{\prime}=\diag(e^{\alpha},e^{\alpha},e^{-\alpha},e^{-\alpha})$ and $v \in Z(K)$ is chosen to be the matrix in $K$ that in the $\U(2)$-representation of $K$ is given by
\begin{equation} \label{eq:v}
  v=\left( \begin{array}{cc} \frac{1}{\sqrt{2}}(1+i) & 0 \\ 0 & \frac{1}{\sqrt{2}}(1+i) \end{array} \right).
\end{equation}
Given a $K$-bi-invariant multiplier on $G$, this map gives rise to a $K_3$-bi-invariant multiplier on $K$. We state the following result, but omit its proof, as it is similar to the one of Lemma \ref{lem:fromGtoK}.
\begin{lem} \label{lem:chialpha}
	Let $\varphi:G \lra \bbC$ be a continuous $K$-bi-invariant function such that $\check{\varphi} \in MS^p(L^2(G))$ for some $p \in (1,\infty)$, and for $\alpha \in \bbR$ let $\tilde{\chi}_{\alpha}:K \lra \bbC$ be defined by $\tilde{\chi}_{\alpha}(k)=\varphi(D_{\alpha}^{\prime}kvD_{\alpha}^{\prime})$. Then $\tilde{\chi}_{\alpha}$ is $K_3$-bi-invariant and satisfies
	\[
		\|\check{\tilde{\chi}}_{\alpha}\|_{MS^p(L^2(K))} \leq \|\check{\varphi}\|_{MS^p(L^2(G))}.
	\]
\end{lem}
Consider the restriction $\chi_{\alpha}=\tilde{\chi}_{\alpha}\vert_{K_2}$, which is a $K_3$-bi-invariant multiplier on $K_2$. It follows that $\chi_{\alpha}(u)=\chi_{\alpha}^0(a^2-b^2+c^2-d^2)$, where $u \in K_2$, and where $a,b,c,d$ are as before, and $\|\check{\chi}_{\alpha}\|_{MS^p(L^2(K_2))} \leq \|\check{\varphi}\|_{MS^p(L^2(G))}$.

Suppose that $\alpha_1 \geq \alpha_2 \geq 0$ and let $D(\alpha_1,\alpha_2)$ be as defined above. Again, if we find an element of the form $D_{\alpha}^{\prime}uvD_{\alpha}^{\prime}$ in $KD(\alpha_1,\alpha_2)K$, where now $u$ has to be an element of $\SU(2)$, we can relate the value of a $K$-bi-invariant multiplier $\varphi$ to the value of the multiplier $\chi_{\alpha}$. This again only works for certain $\alpha_1,\alpha_2 \geq 0$. Consider a general element of $\SU(2)$,
\[
  u=\left( \begin{array}{cc} a+ib & -c+id \\ c+id & a-ib \end{array} \right),
\]
with $a^2+b^2+c^2+d^2=1$. For a proof of the following, see \cite[Lemma 3.15]{haagerupdelaat1}.
\begin{lem} \label{lem:circleseqs}
Let $\alpha \geq 0$ and $\beta \geq \gamma \geq 0$, and let $u,v \in K$ be of the form as in \eqref{eq:uform} and \eqref{eq:v} with respect to the identification of $K$ with $\U(2)$. Then $D_{\alpha}^{\prime}uvD_{\alpha}^{\prime} \in KD(\beta,\gamma)K$ if and only if
\begin{equation} \nonumber
  \begin{cases}
    & \sinh^2 \beta + \sinh^2 \gamma = \sinh^2 (2\alpha), \\  
    & \sinh \beta \sinh \gamma = \frac{1}{2}\sinh^2(2\alpha)|r|,
  \end{cases}
  \end{equation}
where $r=a^2-b^2+c^2-d^2$.
\end{lem}
Now we can combine the results that we obtained for both Gelfand pairs.
\begin{lem} \label{lem:betagamma}
  Let $\beta \geq \gamma \geq 0$. Then the equations
\begin{equation} \label{eq:betagamma}
\begin{split}
   \sinh^2(2s) + \sinh^2s &= \sinh^2 \beta + \sinh^2\gamma, \\
   \sinh(2t)\sinh t &= \sinh \beta \sinh \gamma
\end{split}
\end{equation}
have unique solutions $s=s(\beta,\gamma)$, $t=t(\beta,\gamma)$ in the interval $[0,\infty)$. Moreover,
\begin{equation} \label{eq:betagamma3}
  s \geq \frac{\beta}{4}, \qquad t \geq \frac{\gamma}{2}.
\end{equation}
\end{lem}
A proof of this Lemma can be found in \cite[Lemma 3.16]{haagerupdelaat1}.
\begin{center}
\begin{pspicture}(-1,-1) (8,6.5)
    \pscustom[linestyle=none,fillcolor=lightgray,fillstyle=solid,algebraic]{\psplot{0}{5.5}{0}\psplot{5.5}{0}{x}}
    \psaxes[labels=none,ticks=none]{->}(0,0)(0,0)(5.5,5.5)
    \psline[linewidth=0.5pt](0,0)(5.5,5.5)
    \psline[linewidth=0.5pt](0,0)(5.5,2.75)
    \pscircle*(3,1.5){2pt}
    \pscircle*(4,2){2pt}
    \pscircle*(4.35,1.03){2pt}
    \psplot[plotpoints=1000]{3}{4.35}{ 4.5 x div }
    \psplot[plotpoints=1000,algebraic]{4}{4.355}{ sqrt(20-x^(2)) }
    \rput(6.3,5.5){$\alpha_1=\alpha_2$}
    \rput(6.4,2.75){$\alpha_1=2\alpha_2$}
    \rput(5.9,0){$\alpha_1$}
    \rput(0.4,5.5){$\alpha_2$}
    \rput(2.7,1.75){$(2t,t)$}
    \rput(4.2,2.45){$(2s,s)$}
    \rput(4.92,1.0){$(\beta,\gamma)$}
\end{pspicture}
\end{center}
The figure above shows the relative position of $(\beta,\gamma)$, $(2s,s)$ and $(2t,t)$ as in Lemma \ref{lem:betagammacir} and Lemma \ref{lem:betagammahyp} below. Note that $(\beta,\gamma)$ and $(2s,s)$ lie on a path in the $(\alpha_1,\alpha_2)$-plane of the form $\sinh^2 \alpha_1 + \sinh^2 \alpha_2 = \textrm{ constant}$, and $(\beta,\gamma)$ and $(2t,t)$ lie on a path of the form $\sinh \alpha_1 \sinh \alpha_2 = \textrm{constant}$.
\begin{lem} \label{lem:betagammacir}
  For $p > 4$, there exists a constant $C_3(p) > 0$ (depending only on $p$) such that whenever $\beta \geq \gamma \geq 0$ and $s=s(\beta,\gamma)$ is chosen as in Lemma \ref{lem:betagamma}, then for all $\varphi \in C(K \backslash G \slash K)$ for which $\check{\varphi} \in MS^p(L^2(G))$,
\begin{equation} \nonumber
  |\varphi(D(\beta,\gamma))-\varphi(D(2s,s))| \leq C_3(p) e^{-\frac{\beta-\gamma}{4}(\frac{1}{4}-\frac{1}{p})} \|\check{\varphi}\|_{MS^p(L^2(G))}.
\end{equation}
\end{lem}
\begin{proof}
  Assume first that $\beta - \gamma \geq 8$. Let $\alpha \in [0,\infty)$ be the unique solution to $\sinh^2 \beta + \sinh^2 \gamma=\sinh^2(2\alpha)$, and observe that $2\alpha \geq \beta \geq 2$, so in particular $\alpha > 0$. Define
\[
  r_1=\frac{2\sinh \beta \sinh \gamma}{\sinh^2 \beta+\sinh^2 \gamma} \in [0,1],
\]
and $a_1=\left(\frac{1+r_1}{2}\right)^{\frac{1}{2}}$ and $b_1=\left(\frac{1-r_1}{2}\right)^{\frac{1}{2}}$. Furthermore, put
\[
  u_1=\left(\begin{array}{cc} a_1+ib_1 & 0 \\ 0 & a_1-ib_1 \end{array}\right) \in \SU(2),
\]
and let
\[
  v=\left(\begin{array}{cc} \frac{1}{\sqrt{2}}(1+i) & 0 \\ 0 & \frac{1}{\sqrt{2}}(1+i) \end{array}\right),
\]
as previously defined. We now have $2\sinh \beta \sinh \gamma=\sinh^2(2\alpha)r_1$, and $a_1^2-b_1^2=r_1$, so by Lemma \ref{lem:circleseqs}, we have $D_{\alpha}^{\prime}u_1vD_{\alpha}^{\prime} \in KD(\beta,\gamma)K$. Let $s=s(\beta,\gamma)$ be as in Lemma \ref{lem:betagamma}. Then $s \geq 0$ and $\sinh^2(2s)+\sinh^2s=\sinh^2\beta+\sinh^2\gamma=\sinh^2(2\alpha)$. Put
\[
  r_2=\frac{2\sinh (2s) \sinh s}{\sinh^2(2s)+\sinh^2s} \in [0,1],
\]
and
\[
  u_2=\left(\begin{array}{cc} a_2+ib_2 & 0 \\ 0 & a_2-ib_2 \end{array}\right) \in \SU(2),
\]
where $a_2=\left(\frac{1+r_2}{2}\right)^{\frac{1}{2}}$ and $b_2=\left(\frac{1-r_2}{2}\right)^{\frac{1}{2}}$. Since $a_2^2-b_2^2=r_2$, it follows again by Lemma \ref{lem:circleseqs} that $D_{\alpha}^{\prime}u_2vD_{\alpha}^{\prime} \in KD(2s,s)K$. Now, let $\chi_{\alpha}(u)=\varphi(D_{\alpha}^{\prime}uvD_{\alpha}^{\prime})$ for $u \in K_2 \cong \SU(2)$. Then by Lemma \ref{lem:behaviorsu2} and Lemma \ref{lem:chialpha}, it follows that
\[
  |\chi_{\alpha}(u_1)-\chi_{\alpha}(u_2)|=|\chi_{\alpha}^0(r_1)-\chi_{\alpha}^0(r_2)|\leq \hat{C}(p)|r_1-r_2|^{\frac{1}{4}-\frac{1}{p}}\|\check{\varphi}\|_{MS^p(L^2(G))},
\]
provided that $r_1,r_2 \leq \frac{1}{2}$. Hence, under this assumption, using the $K$-bi-invariance of $\varphi$, we get
\begin{equation} \label{eq:betagammas}
  |\varphi(D(\beta,\gamma))-\varphi(D(2s,s))| \leq \hat{C}(p)|r_1-r_2|^{\frac{1}{4}-\frac{1}{p}}\|\check{\varphi}\|_{MS^p(L^2(G))}.
\end{equation}
Note that $r_1 \leq \frac{2\sinh\beta\sinh\gamma}{\sinh^2\beta}=2\frac{\sinh\gamma}{\sinh\beta}$. Hence, using $\beta\geq\gamma+8\geq\gamma$, we get $r_1 \leq 2\frac{e^{\gamma}(1-e^{-2\gamma})}{e^{\beta}(1-e^{-2\beta})}\leq 2e^{\gamma-\beta}$. In particular, $r_1 \leq 2e^{-8} \leq \frac{1}{2}$. Similarly, $r_2 \leq 2\frac{\sinh s}{\sinh{2s}}=\frac{1}{\cosh s}\leq 2e^{-s}$. By Lemma \ref{lem:betagamma}, equation \eqref{eq:betagamma3}, we obtain that $r_2 \leq 2e^{-\frac{\beta}{4}}\leq 2e^{\frac{\gamma-\beta}{4}} \leq 2e^{-2} \leq \frac{1}{2}$. In particular, \eqref{eq:betagammas} holds, and since $|r_1-r_2| \leq \max\{r_1,r_2\} \leq 2e^{\frac{\gamma-\beta}{4}}$, we have proved that
\begin{equation} \label{eq:betagammas2}
  |\varphi(D(\beta,\gamma))-\varphi(D(2s,s))| \leq \hat{C}(p)2^{\frac{1}{4}-\frac{1}{p}}e^{\frac{\gamma-\beta}{4}(\frac{1}{4}-\frac{1}{p})}\|\check{\varphi}\|_{MS^p(L^2(G))}
\end{equation}
under the assumption that $\beta\geq\gamma+8$. If $\gamma \leq \beta < \gamma+8$, we get from $\|\varphi\|_{\infty} \leq \|\check{\varphi}\|_{MS^p(L^2(G))}$ that $|\varphi(D(\beta,\gamma))-\varphi(D(2s,s))| \leq 2\|\check{\varphi}\|_{MS^p(L^2(G))}$. It follows that
\[
  |\varphi(D(\beta,\gamma))-\varphi(D(2s,s))| \leq C_3(p)e^{\frac{\gamma-\beta}{4}(\frac{1}{4}-\frac{1}{p})}\|\check{\varphi}\|_{MS^p(L^2(G))}
\]
for all $(\beta,\gamma)$ with $\beta \geq \gamma \geq 0$, if for all $p \in (1,\infty)$, we put $C_3(p)=\max\{\hat{C}(p)2^{\frac{1}{4}-\frac{1}{p}},2e^{\frac{1}{2}}\}$.
\end{proof}
\begin{lem} \label{lem:betagammahyp}
  For $p > 12$, there exists a constant $C_4(p) > 0$ (depending only on $p$) such that whenever $\beta \geq \gamma \geq 0$ and $t=t(\beta,\gamma)$ is chosen as in Lemma \ref{lem:betagamma}, then for all $\varphi \in C(K \backslash G \slash K)$ for which $\check{\varphi} \in MS^p(L^2(G))$,
\begin{equation} \nonumber
  |\varphi(D(\beta,\gamma))-\varphi(D(2t,t))| \leq C_4(p)e^{-\frac{\gamma}{4}(\frac{1}{4}-\frac{3}{p})}\|\check{\varphi}\|_{MS^p(L^2(G))}.
\end{equation}
\end{lem}
\begin{proof}
  Let $\beta \geq \gamma \geq 0$. Assume first that $\gamma \geq 2$, and let $\alpha \geq 0$ be the unique solution in $[0,\infty)$ to the equation $\sinh\beta\sinh\gamma=\frac{1}{2}\sinh^2\alpha$, and observe that $\alpha>0$, because $\beta \geq \gamma \geq 2$. Put
\[
  a_1=\frac{\sinh\beta-\sinh\gamma}{\sinh (2\alpha)} \geq 0.
\]
Since $\sinh (2\alpha)=2\sinh \alpha \cosh \alpha \geq 2\sinh^2 \alpha$, we have
\[
  a_1 \leq \frac{\sinh \beta}{\sinh(2\alpha)} \leq \frac{\sinh \beta}{2\sinh^2\alpha}=\frac{1}{4\sinh\gamma}.
\]
In particular, $a_1 \leq \frac{1}{4\gamma} \leq \frac{1}{8}$. Put now $b_1=\sqrt{\frac{1}{2}-a_1^2}$. Then $1-a_1^2-b_1^2=\frac{1}{2}$. Hence, $\sinh \beta \sinh \gamma=\sinh^2 \alpha(1-a_1^2-b_1^2)$ and $\sinh \beta - \sinh \gamma=\sinh(2\alpha)a_1$. Let
\[
  u_1=\left(\begin{array}{cc} a_1+ib_1 & -\frac{1}{\sqrt{2}} \\ \frac{1}{\sqrt{2}} & a_1-ib_1 \end{array}\right) \in \SU(2).
\]
By Lemma \ref{lem:hyperbolaseqs}, we have $D_{\alpha}u_1D_{\alpha} \in KD(\beta,\gamma)K$.

By Lemma \ref{lem:betagamma}, we have $\sinh(2t)\sinh t=\sinh \beta \sinh \gamma = \frac{1}{2}\sinh^2 \alpha$. Moreover, by \eqref{eq:betagamma3}, we have $t \geq \frac{\gamma}{2} \geq 1$. By replacing $(\beta,\gamma)$ in the above calculation with $(2t,t)$, we get that the number
\[
  a_2\frac{\sinh (2t)-\sinh t}{\sinh (2\alpha)} \geq 0,
\]
satisfies
\[
  a_2 \leq \frac{1}{4\sinh t} \leq \frac{1}{4\sinh 1} \leq \frac{1}{4}.
\]
Hence, we can put $b_2=\sqrt{\frac{1}{2}-a_2^2}$ and
\[
  u_2=\left(\begin{array}{cc} a_2+ib_2 & -\frac{1}{\sqrt{2}} \\ \frac{1}{\sqrt{2}} & a_2-ib_2 \end{array}\right).
\]
Then
\begin{equation} \nonumber
\begin{split}
  \sinh(2t)\sinh t &= \sinh^2 \alpha(1-a_2^2-b_2^2),\\
  \sinh(2t)-\sinh t &= \sinh(2\alpha)a_2,
\end{split}
\end{equation}
and $u_2 \in \SU(2)$. Hence, by Lemma \ref{lem:hyperbolaseqs}, $D_{\alpha}u_2D_{\alpha} \in KD(2t,t)K$. Put now $\theta_j=\mathrm{arg}(a_j+ib_j)=\frac{\pi}{2}-\sin^{-1} \left(\frac{a_j}{\sqrt{2}}\right)$ for $j=1,2$. Since $0 \leq a_j \leq \frac{1}{2}$ for $j=1,2$, and since $\frac{d}{dt} \sin^{-1} t = \frac{1}{\sqrt{1-t^2}}\leq\sqrt{2}$ for $t \in [0,\frac{1}{\sqrt{2}}]$, it follows that
\begin{equation} \nonumber
\begin{split}  
|\theta_1-\theta_2| &\leq \bigg\vert\sin^{-1} \left(\frac{a_1}{\sqrt{2}}\right) - \sin^{-1} \left(\frac{a_2}{\sqrt{2}}\right)\bigg\vert \\
  &\leq |a_1-a_2| \\
  &\leq \max\{a_1,a_2\} \\
  &\leq \max\left\{\frac{1}{4\sinh \gamma},\frac{1}{4\sinh t}\right\} \\
  &\leq \frac{1}{4\sinh \frac{\gamma}{2}},
\end{split}
\end{equation}
because $t \geq \frac{\gamma}{2}$. Since $\gamma \geq 2$, we have $\sinh \frac{\gamma}{2}=\frac{1}{2}e^{\frac{\gamma}{2}}(1-e^{-\gamma})\geq\frac{1}{4}e^{\frac{\gamma}{2}}$. Hence, $|\theta_1-\theta_2| \leq e^{-\frac{\gamma}{2}}$. Note that $a_j=\frac{1}{\sqrt{2}}e^{i\theta_j}$ for $j=1,2$, so by Lemma \ref{lem:behavioru2} and Lemma \ref{lem:fromGtoK}, the function $\psi_{\alpha}(u)=\varphi(D_{\alpha}uD_{\alpha})$, $u \in \U(2) \cong K$ satisfies
\begin{equation}
\begin{split}
  |\psi_{\alpha}(u_1)-\psi_{\alpha}(u_2)| &\leq \tilde{C}(p)|\theta_1-\theta_2|^{\frac{1}{8}-\frac{3}{2p}}\|\check{\psi}_{\alpha}\|_{MS^p(L^2(K))} \\
    &\leq \tilde{C}(p)e^{-\frac{\gamma}{4}(\frac{1}{4}-\frac{3}{p})}\|\check{\varphi}\|_{MS^p(L^2(G))}.
\end{split}
\end{equation}
Since $D_{\alpha}u_1D_{\alpha} \in KD(\beta,\gamma)K$ and $D_{\alpha}u_2D_{\alpha} \in KD(2t,t)K$, it follows that
\[
  |\varphi(D(\beta,\gamma))-\varphi(D(2t,t))| \leq \tilde{C}(p)e^{-\frac{\gamma}{4}(\frac{1}{4}-\frac{3}{p})}\|\check{\varphi}\|_{MS^p(L^2(G))}
\]
for all $\gamma \geq 2$. For $\gamma$ satisfying $0 < \gamma \leq 2$, we can instead use that $\|\varphi\|_{\infty} \leq \|\check{\varphi}\|_{MS^p(L^2(G))}$. Hence, for all $p \in (1,\infty)$ putting $C_4(p)=\max\{\tilde{C}(p),2e^{\frac{1}{8}}\}$, we obtain
\[
  |\varphi(D(\beta,\gamma))-\varphi(D(2t,t))| \leq C_4(p)e^{-\frac{\gamma}{4}(\frac{1}{4}-\frac{3}{p})}\|\check{\varphi}\|_{MS^p(L^2(G))}
\]
for all $\beta \geq \gamma \geq 0$.
\end{proof}
For a proof of the following lemma, see \cite[Lemma 3.19]{haagerupdelaat1}.
\begin{lem} \label{lem:rhosigma}
  Let $s \geq t \geq 0$. Then the equations
\begin{equation} \label{eq:system1}
\begin{split}
  \sinh^2 \beta + \sinh^2 \gamma &= \sinh^2(2s)+\sinh^2 s, \\
  \sinh \beta \sinh \gamma &= \sinh(2t)\sinh t,
\end{split}
\end{equation}
have a unique solution $(\beta,\gamma) \in \bbR^2$ for which $\beta \geq \gamma \geq 0$. Moreover, if $1 \leq t \leq s \leq \frac{3t}{2}$, then
\begin{equation} \label{eq:system2}
\begin{split}
  |\beta-2s| &\leq 1, \\
  |\gamma+2s-3t| &\leq 1.
\end{split}
\end{equation}
\end{lem}
\begin{lem} \label{lem:comparest}
  For all $p > 12$, there exists a constant $C_5(p) > 0$ such that whenever $s,t \geq 0$ satisfy $2 \leq t \leq s \leq \frac{6}{5}t$, then for all $\varphi \in C(K \backslash G \slash K)$ for which $\check{\varphi} \in MS^p(L^2(G))$,
\[
  |\varphi(D(2s,s))-\varphi(D(2t,t))| \leq C_5(p) e^{-\frac{s}{8}(\frac{1}{4}-\frac{3}{p})} \|\check{\varphi}\|_{MS^p(L^2(G))}.
\]
\end{lem}
\begin{proof}
  Choose $\beta \geq \gamma \geq 0$ as in Lemma \ref{lem:rhosigma}. Then by Lemma \ref{lem:betagammacir} and Lemma \ref{lem:betagammahyp}, we have for $p > 12$,
\begin{equation} \nonumber
\begin{split}
  |\varphi(D(2s,s))-\varphi(D(\beta,\gamma))| &\leq C_3(p)e^{-\frac{\beta-\gamma}{4}(\frac{1}{4}-\frac{1}{p})}\|\check{\varphi}\|_{MS^p(L^2(G))}, \\
  |\varphi(D(2t,t))-\varphi(D(\beta,\gamma))| &\leq C_4(p)e^{-\frac{\gamma}{4}(\frac{1}{4}-\frac{3}{p})}\|\check{\varphi}\|_{MS^p(L^2(G))}.
\end{split}
\end{equation}
Moreover, by \eqref{eq:system2},
\begin{equation} \nonumber
\begin{split}
  \beta - \gamma &\geq (2s-1) - (3t-2s+1) = 4s - 3t -2 \geq s-2, \\
  \gamma &\geq 3t-2s-1 \geq \frac{5}{2}s-2s-1=\frac{s-2}{2}.
\end{split}
\end{equation}
Hence, since $s \geq 2$, we have $\min\{e^{-\gamma},e^{-(\beta-\gamma)}\} \leq e^{-\frac{s-2}{2}}$. Thus, the lemma follows from Lemma \ref{lem:betagammacir} and Lemma \ref{lem:betagammahyp} with $C_5(p)=e^{\frac{1}{16}}(C_3(p)+C_4(p))$.
\end{proof}
\begin{lem} \label{lem:limit}
  For $p > 12$, there exists a constant $C_6(p) > 0$ such that for all $\varphi \in C(K \backslash G \slash K)$ for which $\check{\varphi} \in MS^p(L^2(G))$, the limit $c_{\infty}(\varphi)=\lim_{t \to \infty} \varphi(D(2t,t))$ exists, and for all $t \geq 0$,
\[
  |\varphi(D(2t,t))-c_{\infty}(\varphi)| \leq C_6(p)e^{-\frac{t}{8}(\frac{1}{4}-\frac{3}{p})}\|\check{\varphi}\|_{MS^p(L^2(G))}.
\]
\end{lem}
\begin{proof}
  By Lemma \ref{lem:comparest}, we have for $u \geq 5$ and $\gamma \in [0,1]$, that
\begin{equation} \label{eq:ugamma}
  |\varphi(D(2u,u))-\varphi(D(2u+2\gamma,u+\gamma))| \leq C_5(p)e^{-\frac{u}{8}(\frac{1}{4}-\frac{3}{p})}\|\check{\varphi}\|_{MS^p(L^2(G))},
\end{equation}
since $u \leq u + \gamma$. Let $s \geq t \geq 5$. Then $s=t+n+\delta$, where $n \geq 0$ is an integer and $\delta \in [0,1)$. Applying equation \eqref{eq:ugamma} to $(u,\gamma)=(t+j,1)$, $j=0,1,\ldots,n-1$ and $(u,\gamma)=(t+n,\delta)$, we obtain
\begin{equation} \nonumber
\begin{split}
 |\varphi(D(2t,t))-\varphi(D(2s,s))| &\leq C_5(p)\left(\sum_{j=0}^n e^{-\frac{t+j}{8}(\frac{1}{4}-\frac{3}{p})}\right) \|\check{\varphi}\|_{MS^p(L^2(G))} \\ &\leq C_5(p)^{\prime}e^{-\frac{t}{8}(\frac{1}{4}-\frac{3}{p})}\|\check{\varphi}\|_{MS^p(L^2(G))},
\end{split}
\end{equation}
where $C_5^{\prime}(p)=C_5(p)\sum_{j=0}^{\infty} e^{-\frac{j}{8}(\frac{1}{4}-\frac{3}{p})}$. Hence, $(\varphi(D(2t,t)))_{t \geq 5}$ is a Cauchy net. Therefore, $c_{\infty}(\varphi)=\lim_{t \to \infty} \varphi(D(2t,t))$ exists, and
\[
  |\varphi(D(2t,t))-c_{\infty}(\varphi)|=\lim_{s \to \infty} |\varphi(D(2t,t))-\varphi(D(2s,s))| \leq C_5^{\prime}(p)e^{-\frac{t}{8}(\frac{1}{4}-\frac{3}{p})}\|\check{\varphi}\|_{MS^p(L^2(G))}
\]
for all $t \geq 5$. Since $\|\varphi\|_{\infty} \leq \|\check{\varphi}\|_{MS^p(L^2(G))}$, we have for all $0 \leq t < 5$,
\[
  |\varphi(D(2t,t))-c_{\infty}(\varphi)| \leq 2\|\check{\varphi}\|_{MS^p(L^2(G))}.
\]
Hence, the lemma follows with $C_6(p)=\max\{C_5^{\prime}(p),2e^{\frac{5}{32}}\}$.
\end{proof}
\begin{proof}[Proof of Proposition \ref{prp:sp2abapschur}]
Let $\varphi \in C(K \backslash G \slash K)$ be such that $\check{\varphi} \in MS^p(L^2(G))$, and let $(\alpha_1,\alpha_2)=(\beta,\gamma)$, where $\beta \geq \gamma \geq 0$. Assume first $\beta \geq 2\gamma$. Then $\beta - \gamma \geq \frac{\beta}{2}$, so by Lemma \ref{lem:betagamma} and Lemma \ref{lem:betagammacir}, there exists an $s \geq \frac{\beta}{4}$ such that
\[
  |\varphi(D(\beta,\gamma))-\varphi(D(2s,s))| \leq C_3(p)e^{-\frac{\beta}{8}(\frac{1}{4}-\frac{1}{p})}\|\check{\varphi}\|_{MS^p(L^2(G))}.
\]
By Lemma \ref{lem:limit},
\begin{equation} \nonumber
\begin{split}
  |\varphi(D(2s,s))-c_{\infty}(\varphi)| &\leq C_6(p)e^{-\frac{s}{8}(\frac{1}{4}-\frac{3}{p})}\|\check{\varphi}\|_{MS^p(L^2(G))} \\
    &\leq C_6(p)e^{-\frac{\beta}{32}(\frac{1}{4}-\frac{3}{p})}\|\check{\varphi}\|_{MS^p(L^2(G))}.
\end{split}
\end{equation}
Hence,
\[
  |\varphi(D(\beta,\gamma))-c_{\infty}(\varphi)| \leq (C_3(p)+C_6(p))e^{-\frac{\beta}{32}(\frac{1}{4}-\frac{3}{p})}\|\check{\varphi}\|_{MS^p(L^2(G))}.
\]
Assume now that $\beta < 2\gamma$. Then, by Lemma \ref{lem:betagamma} and Lemma \ref{lem:betagammahyp}, we obtain that there exists a $t \geq \frac{\gamma}{2} > \frac{\beta}{4}$ such that
\[
  |\varphi(D(\beta,\gamma))-\varphi(D(2t,t))| \leq C_4(p)e^{-\frac{\beta}{8}(\frac{1}{4}-\frac{3}{p})}\|\check{\varphi}\|_{MS^p(L^2(G))},
\]
and again by Lemma \ref{lem:limit},
\begin{equation} \nonumber
\begin{split}
  |\varphi(D(2t,t))-c_{\infty}(\varphi)| &\leq C_6(p)e^{-\frac{t}{8}(\frac{1}{4}-\frac{3}{p})}\|\check{\varphi}\|_{MS^p(L^2(G))} \\
    &\leq C_6(p)e^{-\frac{\beta}{32}(\frac{1}{4}-\frac{3}{p})}\|\check{\varphi}\|_{MS^p(L^2(G))}.
\end{split}
\end{equation}
Hence,
\[
  |\varphi(D(\beta,\gamma))-c_{\infty}(\varphi)| \leq (C_4(p)+C_6(p))e^{-\frac{\beta}{32}(\frac{1}{4}-\frac{3}{p})}\|\check{\varphi}\|_{MS^p(L^2(G))}.
\]
Combining these results, and using that $\|\alpha\|_2 = \sqrt{\beta^2+\gamma^2} \leq \sqrt{2}\beta$, it follows that for all $\beta \geq \gamma \geq 0$,
\[
  |\varphi(D(\beta,\gamma))-c_{\infty}(\varphi)| \leq C_1(p)e^{-C_2(p)\|\alpha\|_2}\|\check{\varphi}\|_{MS^p(L^2(G))},
\]
where $C_1(p)=\max\{C_3(p)+C_6(p),C_4(p)+C_6(p)\}$ and $C_2(p)=\frac{1}{32\sqrt{2}}(\frac{1}{4}-\frac{3}{p})$. This proves the proposition.
\end{proof}
The values $p \in [1,\frac{12}{11}) \cup (12,\infty]$ give sufficient conditions for $\Sp(2,\mathbb{R})$ to fail the $\apschur$. We would like to point out that the set of these values might be bigger.

\section{Noncommutative $L^p$-spaces without the OAP} \label{sec:general}
In the previous section we proved that $\Sp(2,\bbR)$ does not have the $\apschur$ for $p \in [1,\frac{12}{11}) \cup (12,\infty]$. By Lemma \ref{lem:oapapschur}, this directly implies the following theorem.
\begin{thm}
  Let $p \in [1,\frac{12}{11}) \cup (12,\infty]$, and let $\Gamma$ be a lattice in $\Sp(2,\bbR)$. Then the noncommutative $L^p$-space $L^p(L(\Gamma))$ does not have the OAP (or CBAP).
\end{thm}
Combining Theorem \ref{thm:sp2notapschur} and Theorem \ref{thm:ldlsmain}, this implies the following result.
\begin{thm} \label{thm:maintheoremapschur}
  Let $p \in [1,\frac{12}{11}) \cup (12,\infty]$, and let $G$ be a connected simple Lie group with finite center and real rank greater than or equal to two. Then $G$ does not have the $\apschur$.
\end{thm}
\begin{proof}
Let $G$ be a connected simple Lie group with finite center and real rank greater than or equal to two. By Wang's method \cite{wang}, we may assume that $G$ is the adjoint group, so that $G$ has a connected semisimple subgroup $H$ with real rank $2$. Such a subgroup is closed, as was proved in \cite{dorofaeff}. It is known that $H$ has finite center and is locally isomorphic to either $\SL(3,\bbR)$ or $\Sp(2,\bbR)$ \cite{boreltits}, \cite{margulis}. Since the $\apschur$ passes to closed subgroups and is preserved under local isomorphisms (see Proposition \ref{prp:apschurproperties}), we conclude that $G$ does not have the $\apschur$ for $p \in [1,\frac{12}{11}) \cup (12,\infty]$, since both $\SL(3,\bbR)$ and $\Sp(2,\bbR)$ do not have the $\apschur$ for such $p$.
\end{proof}
Combining this result with Proposition \ref{thm:apschurlattices} and Lemma \ref{lem:oapapschur}, we obtain the main theorem of this article.
\begin{thm} \label{thm:maintheoremnclp}
  Let $p \in [1,\frac{12}{11}) \cup (12,\infty]$, and let $\Gamma$ be a lattice in a connected simple Lie group with finite center and real rank greater than or equal to two. Then $L^p(L(\Gamma))$ does not have OAP (or CBAP).
\end{thm}

\section*{Acknowledgements}
I thank Uffe Haagerup and Magdalena Musat for valuable discussions and useful suggestions and remarks.

\bibliographystyle{amsplain}

\end{document}